\documentclass{amsart}
\usepackage[cp1251]{inputenc}
\usepackage{amsmath,amssymb,amsthm}
\usepackage[mathscr]{eucal}
\usepackage[dvips]{graphicx}

\newtheorem*{theorem}{Theorem}
\newtheorem*{proposition}{Proposition}
\newtheorem*{lemma}{Lemma}
\newtheorem*{definition}{Definition}
\newtheorem{example}{Example}

\author{Evgeny Fominykh}

\address{Chelyabinsk State University,
Chelyabinsk, 454001, Russia}

\address{Institute of Mathematics and Mechanics,
Ural Branch of the Russian Academy of Sciences, Ekaterinburg, 620219, Russia}
 \email{fominykh@csu.ru}

\title{DEHN SURGERIES ON THE FIGURE EIGHT KNOT: AN UPPER BOUND FOR THE COMPLEXITY}

\subjclass[2000]{Primary: 57M99, Secondary: 57M25.}

\keywords{Dehn surgery, figure eight knot, upper bound of the
complexity.}

\thanks{The research was supported by the Russian Foundation for Basic Research (grant 10-01-91056)
and by the Joint Program of the Institute for Mathematics and
Mechanics UrO RAN and of the Institute for Mathematics SO RAN}

\begin{document}

\begin{abstract}
 We establish an upper bound $\omega(p/q)$ on the complexity of
manifolds obtained by $p/q$-surgeries on the figure eight knot. It
turns out that if $\omega(p/q)\leqslant 12$, the bound is sharp.
\end{abstract}

\maketitle

\markright{Dehn surgeries on the figure eight knot: an upper bound
for the complexity}

\section*{Introduction}

 The notion of the complexity $c(M)$ of a compact $3$-manifold $M$
was introduced in \cite{Matveev-1990}. The complexity is defined as
the minimal possible number of true vertices of an almost simple
spine of $M$. If $M$ is closed and irreducible and $c(M)>0$, then
$c(M)$ is the minimal number of tetrahedra needed to obtain $M$ by
gluing together their faces. The problem of calculating the
complexity $c(M)$ is very difficult. The exact values of the
complexity are presently known only for certain infinite series of
irreducible boundary irreducible $3$-manifolds \cite{FMP,
Anisov-2005, Jaco-Rubinstein-Tillmann-2009}. In addition, this
problem is solved for all closed orientable irreducible manifolds up
to complexity $12$ (see \cite{Matveev-2005}). Note that the table
given in \cite{Matveev-2005} contains $36833$ manifolds and is only
available in electronic form \cite{Atlas}.

 The task of finding an upper bound for the complexity of a manifold
$M$ does not present any particular difficulties. To do that it
suffices to construct an almost simple spine $P$ of $M$. The number
of true vertices of $P$ will serve as an upper bound for the
complexity. It is known \cite[2.1.2]{Matveev-2003} that an almost
simple spine can be easily constructed from practically any
representation of a manifold. The rather large number of manifolds
in \cite{Atlas} gives rise to a new task of finding potentially
sharp upper bounds for the complexity, i.e. upper bounds that would
yield the exact value of the complexity for all manifolds from the
table \cite{Atlas}. An important result in this direction was
obtained by Martelli and Petronio \cite{Martelli-Petronio-2004}.
They found a potentially sharp upper bound for the complexity of all
closed orientable Seifert manifolds. Similar results for infinite
families of graph manifolds can be found in
\cite{Fominykh-Ovchinnikov-2005, Fominykh-2008}.

 An upper bound $h(r/s, t/u, v/w)$ for the complexity of hyperbolic
manifolds obtained by surgeries on the link $6^3_1$ (in Rolfsen's
notation \cite{Rolfsen-1976}) with rational parameters $(r/s, t/u,
v/w)$ is given by Martelli and Petronio in
\cite{Martelli-Petronio-2004}. It turns out that the bound is not
sharp for a large number of manifolds, as the following two examples
show. First, the value of $h$ is equal to $10$ only for $13$ of $24$
manifolds of complexity $10$ obtained by surgeries on $6^3_1$ (see
\cite{Martelli-Petronio-2004}). Second, on analyzing the table
\cite{Atlas} we noticed that the bound is not sharp for $44$ of $46$
manifolds of the type $6^3_1(1, 2, v/w)$ with complexity less or
equal to $12$. Denote by $4_1(p/q)$ the closed orientable
$3$-manifold obtained from the figure eight knot $4_1$ by
$p/q$-surgery. Since the manifolds $4_1(p/q)$ and $6^3_1(1, 2,
p/q+1)$ are homeomorphic, a potentially sharp upper bound for the
complexity of such manifolds become important.

 The following theorem is the main result of the paper. To give an
exact formulation, we need to introduce a certain
$\mathbb{N}$-valued function $\omega(p/q)$ on the set of
non-negative rational numbers. Let $p\geqslant 0$, $q\geqslant 1$ be
relatively prime integers, let $[p/q]$ be the integer part of $p/q$,
and let $rem(p,q)$ be the remainder of the division of $p$ by $q$.
As in \cite{Matveev-2003}, we denote by $S(p,q)$ the sum of all
partial quotients in the expansion of $p/q$ as a regular continued
fraction. Now we define:
 $$\omega(p/q) = a(p/q) + \max\{[p/q]-3, 0\} + S(rem(p,q),q),$$
where
 $$a(p/q) = \left\{%
\begin{array}{ll}
    6, & \hbox{if } p/q=4, \\
    7, & \hbox{if } p/q\in \mathbb{Z} \hbox{ and } p/q\neq 4,\\
    8, & \hbox{if } p/q\not\in \mathbb{Z}.\\
\end{array}%
\right.$$

\begin{theorem}
 For any two relatively prime integers $p\geqslant 0$ and
$q\geqslant 1$ we have the inequality $c(4_1(p/q))\leqslant
\omega(p/q)$. Moreover, if $\omega(p/q)\leqslant 12$, then
$c(4_1(p/q)) = \omega(p/q)$.
\end{theorem}

 Note that the restrictions $p\geqslant 0$ and $q\geqslant 1$ in the
above theorem are inessential, since the knot $4_1$ is equivalent to
its mirror image, which implies $4_1(-p/q)$ is homeomorphic to
$4_1(p/q)$.


\section{Preliminaries}

 In this section we recall some known definitions and facts that
will be used in the paper.

\subsection{Theta-curves on a torus}

 By a theta-curve $\theta\subset T$ on a torus $T$ we mean a graph
that is homeomorphic to a circle with a diameter and such that
$T\setminus \theta$ is an open disc. It is well known
\cite{Martelli-Petronio-2004, Anisov-1994} that any two theta-curves
on $T$ can be transformed into each other by isotopies and by a
sequence of flips (see Fig.~\ref{flip-transformation}). Let us endow
the set $\Theta(T)$ of theta-curves on $T$ with the distance
function $d$ defining for given $\theta, \theta'\in \Theta(T)$ the
distance $d(\theta, \theta')$ between them as the minimal number of
flips required to transform $\theta$ into $\theta'$.

\begin{figure}
\centering
\includegraphics[scale=0.6]{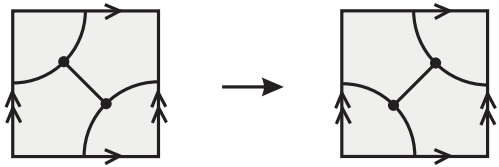}
\caption{A flip-transformation} \label{flip-transformation}
\end{figure}

 For calculating the distance between two theta-curves on a torus
we use the classical ideal triangulation $\mathbb{F}$ (Farey
tesselation) of the hyperbolic plane $\mathbb{H}^2$. If we view the
hyperbolic plane $\mathbb{H}^2$ as the upper half plane of
$\mathbb{C}$ bounded by the circle $\partial \mathbb{H}^2 =
\mathbb{R}\cup \{\infty\}$, then the triangulation $\mathbb{F}$ has
vertices at the points of $\mathbb{Q}\cup \{1/0\}\subset \partial
\mathbb{H}^2$, where $1/0=\infty$, and its edges are all the
geodesics in $\mathbb{H}^2$ with endpoints the pairs $a/b$, $c/d$
such that $ad-bc=\pm 1$. For convenience, the images of the
hyperbolic plane $\mathbb{H}^2$ and of the triangulation
$\mathbb{F}$ under the mapping $z\to (z-i)/(z+i)$ are shown in
Fig.~\ref{triangulation}.

\begin{figure}
\centering
\includegraphics[scale=0.5]{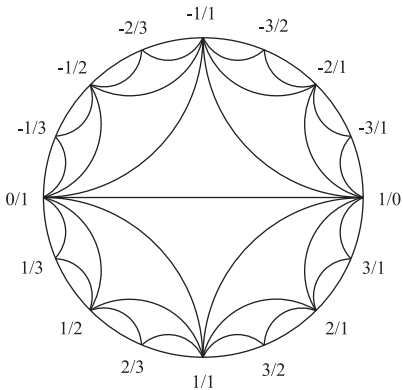}
\caption{The ideal Farey triangulation of the hyperbolic plane}
\label{triangulation}
\end{figure}

 Fix some coordinate system $(\mu, \lambda)$ on a torus $T$. We now
construct a map $\Psi_{\mu, \lambda}$ from $\Theta(T)$ to the set of
triangles of $\mathbb{F}$. To do that we consider the map
$\psi_{\mu, \lambda}$ that assigns to each nontrivial simple closed
curve $\mu^{\alpha}\lambda^{\beta}$ on $T$ the point $\alpha/
\beta\in \partial\mathbb{H}^2$. Note that each theta-curve $\theta$
on $T$ contains three nontrivial simple closed curves $\ell_1$,
$\ell_2$, $\ell_3$, that are formed by the pairs of edges of
$\theta$. Since the intersection index of every two curves $\ell_i$,
$\ell_j$, $i\neq j$, is equal to $\pm 1$, the points $\psi_{\mu,
\lambda}(\ell_1)$, $\psi_{\mu, \lambda}(\ell_2)$, $\psi_{\mu,
\lambda}(\ell_3)$ are the vertices of a triangle $\triangle$ of the
Farey triangulation, and we define $\Psi_{\mu, \lambda}(\theta)$ to
be $\triangle$.

 Denote by $\Sigma$ the graph dual to the triangulation
$\mathbb{F}$. This graph is a tree because the triangulation is
ideal. We now define the distance between any two triangles of
$\mathbb{F}$ to be the number of edges of the only simple path in
$\Sigma$ that joins the corresponding vertices of the dual graph.
The key observation used for the practical calculations is that for
any coordinate system $(\mu, \lambda)$ on $T$ the distance between
any two theta-curves $\theta$, $\theta'$ is equal to the distance
between the triangles $\Psi_{\mu, \lambda}(\theta)$, $\Psi_{\mu,
\lambda}(\theta')$ of the Farey triangulation. The reason is that if
$\theta'$ is obtained from $\theta$ via a flip, the corresponding
triangles have a common edge.

\subsection{Simple and special spines}

 A compact polyhedron $P$, following Matveev \cite{Matveev-2003}, is
called simple if the link of each point $x\in P$ is homeomorphic to
one of the following $1$-dimensional polyhedra:
\begin{itemize}
    \item[(a)] a circle (the point $x$ is then called nonsingular);
    \item[(b)] a circle with a diameter (then $x$ is a triple point);
    \item[(c)] a circle with three radii (then $x$ is a true vertex).
\end{itemize}

 The components of the set of nonsingular points are said to
be the $2$-components of $P$, while the components of the set of
triple points are said to be the triple lines of $P$. A simple
polyhedron is special if each of its triple lines is an open
$1$-cell and each of its $2$-components is an open $2$-cell.

 A subpolyhedron $P$ of a $3$-manifold $M$ is a spine of $M$ if
$\partial M\neq\emptyset$ and the manifold $M\setminus P$ is
homeomorphic to $\partial M \times (0,1]$, or $\partial M=\emptyset$
and $M\setminus P$ is an open ball. A spine of a $3$-manifold is
called simple or special if it is a simple or special polyhedron,
respectively.


\subsection{Relative spines}

 A manifold with boundary pattern, following Johannson
\cite{Johannson-1979}, is a $3$-manifold $M$ with a fixed graph
$\Gamma \subset \partial M$ that does not have any isolated
vertices. A manifold $M$ with boundary pattern $\Gamma$ can be
conveniently viewed as a pair $(M, \Gamma)$. The case $\Gamma =
\emptyset$ is also allowed.

\begin{definition}
 Let $(M, \Gamma)$ be a $3$-manifold with boundary pattern. Then a
subpolyhedron $P\subset M$ is called a relative spine of $(M,
\Gamma)$ if the following holds:
 \begin{enumerate}
 \item $M\setminus P$ is an open ball;
 \item $\partial M \subset P$;
 \item $\partial M \cap Cl(P\setminus\partial M) = \Gamma$.
 \end{enumerate}
\end{definition}

A relative spine is simple if it is a simple polyhedron.
Obviously, if $M$ is closed, then any relative spine of $(M,
\emptyset)$ is a spine of $M$.

\begin{figure}
\centering
\includegraphics[scale=0.6]{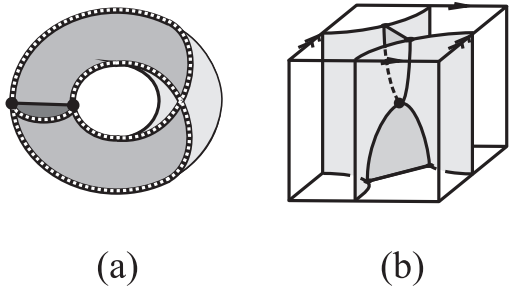}
\caption{Examples of simple relative spines} \label{blocks}
\end{figure}

\begin{example}
 Let $V$ be a solid torus with a meridian $m$. Choose a
simple closed curve $\ell$ on $\partial V$ that intersects $m$ twice
in the same direction. Note that $\ell$ decomposes $m$ into two
arcs. Consider a theta-curve $\theta_V\subset \partial V$ consisting
of $\ell$ and an arc (denote it by $\gamma$) of $m$. Then the
manifold $(V, \theta_V)$ has a simple relative spine without
interior true vertices. This spine is the union of $\partial V$, a
M\"{o}bius strip inside V, and a part of meridional disc bounded by
$\gamma$ (Figure~\ref{blocks}a).

 Note that among the three nontrivial simple closed curves contained in
$\theta_V$, none is isotopic to the meridian $m$ of $V$. On the
other hand, applying the flip to $\theta_V$ along $\gamma$, we get a
theta-curve $\theta_m\subset \partial V$ containing $m$.
\end{example}

\begin{example}
 Let $\theta$, $\theta'$ be two theta-curves on a torus $T$ such that
$\theta'$ is obtained from $\theta$ by exactly one flip. Then the
manifold
 $$(T\times [0, 1], (\theta\times \{0\})\cup (\theta'\times \{1\}))$$
has a simple relative spine $R$ with one interior true vertex (in
Figure~\ref{blocks}b the torus $T$ is represented as a square with
the sides identified). Note that $R$ satisfies the following
conditions:
 \begin{itemize}
 \item[(1)] for each $t\in [0, 1/2)$ a theta-curve $\theta_t$, where
 $R\cap (T\times \{t\}) = \theta_t\times \{t\}$, is isotopic to
 $\theta$;
 \item[(2)] for each $t\in (1/2, 1]$ the theta-curve $\theta_t$ is isotopic to
 $\theta'$;
 \item[(3)] $R\cap (T\times \{1/2\})$ is a wedge of two circles.
 \end{itemize}
\end{example}

\subsection{Assembling of manifolds with boundary patterns}

 Denote by $\mathscr{T}$ the class of all manifolds $(M,
\Gamma)$ such that any component $T$ of $\partial M$ is a torus and
$T \cap \Gamma$ is a theta-curve. Let $(M, \Gamma)$ and $(M',
\Gamma')$ be two manifolds in $\mathscr{T}$ with nonempty
boundaries. Choose two tori $T\subseteq \partial M$, $T'\subseteq
\partial M'$ and a homeomorphism $\varphi: T\to T'$ taking the theta-curve
$\theta = T\cap \Gamma$ to the theta-curve $\theta' = T'\cap
\Gamma'$. Then we can construct a new manifold $(W, \Xi)\in
\mathscr{T}$, where $W = M\cup_\varphi M'$, and $\Xi =
(\Gamma\setminus \theta)\cup (\Gamma'\setminus \theta')$. In this
case we say that the manifold $(W, \Xi)$ is obtained assembling $(M,
\Gamma)$ and $(M', \Gamma')$ \cite{Martelli-Petronio-2001}.

 Note that if manifolds $(M, \Gamma)$ and $(M', \Gamma')$ have
simple relative spines denoted $P$ and $P'$ respectively, with $v$
and $v'$ interior true vertices, then the manifold $(W, \Xi)$ has a
simple relative spine $R$ with $v + v'$ interior true vertices.
Indeed, $R$ can be obtained by gluing $P$ and $P'$ along $\varphi$
and removing the open disc in $P\cup_\varphi P'$ that is obtained by
identifying $T\setminus \theta$ with $T'\setminus \theta'$.

 To prove the main theorem of the paper we generalize the notion of the assembling
by removing the restriction $\varphi(\theta) = \theta'$.

\begin{lemma}
 \label{assembling}
 Let $(M, \Gamma)$ and $(M', \Gamma')$ be two manifolds in
$\mathscr{T}$ with nonempty boundaries that admit simple relative
spines with $v$ and $v'$ interior true vertices respectively. Then
for any homeomorphism $\varphi: T\to T'$ of a torus $T\subseteq
\partial M$ onto a torus $T'\subseteq \partial M'$ there exists a simple
relative spine of a manifold $(W, \Upsilon)$, where $W =
M\cup_\varphi M'$ and $\Upsilon = (\Gamma\setminus \theta)\cup
(\Gamma'\setminus \theta')$, with $v + v' + d(\varphi(\theta),
\theta')$ interior true vertices.
\end{lemma}

\begin{proof}
 First, by induction on the number $n = d(\varphi(\theta), \theta')$
we prove that there exists a simple relative spine of the manifold
 $$(M'', \Gamma'') = (T'\times [0, 1], (\varphi(\theta)\times \{0\})\cup
 (\theta'\times \{1\}))$$
with $n$ interior true vertices. If $n=0$, i.e. the theta-curve
$\varphi(\theta)$ is isotopic to the theta-curve $\theta'$, the
desired spine is isotopic to the polyhedron
$(\varphi(\theta)\times [0, 1])\cup \partial M''$. Suppose that $n
> 0$. As has already been alluded to in the beginning of the
section 1.1, there exists a sequence $\{\theta_i\}_{i=0}^n$ of
pairwise distinct theta-curves on the torus $T'$ such that
$\theta_0 = \varphi(\theta)$, $\theta_n = \theta'$, and $\theta_i$
is obtained from $\theta_{i-1}$ by a flip, for $i=1\ldots n$. The
induction assumption implies that the manifold
 \begin{equation}
 \label{m1}
 (T'\times [0, 1/2], (\theta_0\times \{0\})\cup
 (\theta_{n-1}\times \{1/2\}))
 \end{equation}
has a simple relative spine with $n-1$ interior true vertices.
Furthermore, the simple relative spine of the manifold
 \begin{equation}
 \label{m2}
 (T'\times [1/2, 1], (\theta_{n-1}\times \{1/2\})\cup
 (\theta_n\times \{1\}))
 \end{equation}
with one interior true vertex is described in the Example 2. Then
the desired spine of the manifold $(M'', \Gamma'')$ is obtained by
assembling the manifolds (\ref{m1}) and (\ref{m2}) along the
identity map on $T'\times \{1/2\}$.

 Now, note that the consecutive assemblings of the manifolds $(M,
\Gamma)$, $(M'', \Gamma'')$ and $(M', \Gamma')$ along natural
homeomorphisms that take each point $x\in T$ to the point
$(\varphi(x), 0)\in T'\times \{0\}$, and each point $(y, 1)\in
T'\times \{1\}$ to the point $y\in T'$, yield the manifold $(W,
\Upsilon)$ and its simple relative spine with $v + v' +
d(\varphi(\theta), \theta')$ interior true vertices.
\end{proof}


\section{Relative spines of the figure eight knot complement}

 In this section we construct some simple relative spines of
the figure eight knot complement $E(4_1)$. Let us fix a canonical
coordinate system on the boundary torus $\partial E(4_1)$ consisting
of oriented closed curves $\mu$, $\lambda$ such that the meridian
$\mu$ generates $H_1(E(4_1); \mathbb{Z})$ and the longitude
$\lambda$ bounds a surface in $E(4_1)$. This system determines the
map $\Psi_{\mu, \lambda}$ from $\Theta(T)$ to the set of triangles
of the Farey triangulation. Denote by $\triangle^{(i)}$ the triangle
of $\mathbb{F}$ with the vertices at $i$, $i+1$, and $\infty$.

\begin{proposition}
 \label{spine}
 For any $i\in\{ 0, 1, 2, 3\}$ there exists a theta-curve
$\theta^{(i)}$ on the torus $\partial E(4_1)$ such that the
manifold $(E(4_1), \theta^{(i)})$ has a simple relative spine with
$10$ interior true vertices and $\Psi_{\mu, \lambda}(\theta^{(i)})
= \triangle^{(i)}$.
\end{proposition}

\begin{proof}
 Step 1. Let $P$ be a special spine of an arbitrary compact orientable
$3$-manifold $M$ whose boundary is a torus, and let $\theta$ be a
theta-curve on $\partial M$. We begin the proof by describing a
method for constructing a simple relative spine $R(P, \theta)$ of
the manifold $M$.

 By Theorem 1.1.7 \cite{Matveev-2003}, $M$ can be identified with
the mapping cylinder of a local embedding $f:\partial M\to P$.
Denote by $f_{|\theta}:\theta\to P$ the restriction to $\theta$ of
the map $f$. Then the union $R(P, \theta)$ of the mapping cylinder
of $f_{|\theta}$ and of $\partial M$ is a relative spine of $M$,
since $\partial M \subset R(P, \theta)$, $\partial M \cap Cl(R(P,
\theta)\setminus\partial M) = \theta$, and $M\setminus R(P, \theta)$
is homeomorphic to the direct product of the open disc $\partial
M\setminus \theta$ with an interval. In general, $R(P, \theta)$ just
constructed is not necessarily a simple polyhedron. This can be
dealt with by introducing the notion of general position. We say
that a theta-curve $\theta\subset \partial M$ is in general position
with respect to the map $f$, if the image $f(\theta)$ satisfies the
following conditions.
\begin{enumerate}
    \item $f(\theta)$ contains no true vertices of $P$.
    \item For any intersection point $x$ of $f(\theta)$ with
    the triple lines of $P$ there exists a neighborhood $U(x)\subset P$
    such that the intersection $U(x)\cap f(\theta)$ is an arc
    meeting the set of the triple lines of $P$ transversally exactly at $x$.
    \item  For any intersection point $x$ of the set $f(\theta)$ with
    the $2$-components of $P$ its inverse image $f^{-1}_{|\theta}(x)$ consists of at most
    two points of $\theta$. Moreover, if $f^{-1}_{|\theta}(x)$ consists of exactly
    two points, then there exists a neighborhood $U(x)\subset P$ such that
    the inverse image $f^{-1}_{|\theta}(U(x)\cap f(\theta))$ of the
    intersection $U(x)\cap f(\theta)$ is the disjoint union of two
    arcs $\gamma_1$, $\gamma_2$ of $\theta$, and the images $f(\gamma_1)$, $f(\gamma_2)$
    intersect each other transversally at exactly one point $x$.
    Such a point $x$ is called the self-intersection point of the image $f(\theta)$
    of $\theta$.
\end{enumerate}

Obviously, if a theta-curve $\theta$ is in general position with
respect to the map $f$, then the relative spine $R(P, \theta)$ of
the manifold $M$ is simple.

\begin{figure}
\centering
\includegraphics[scale=0.6]{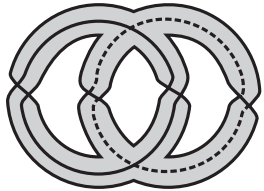}
\caption{A minimal spine of the complement of the figure eight knot}
\label{minspine}
\end{figure}

 Step 2. We consider now the minimal special spine $P$ of the manifold
$M = E(4_1)$ shown in Figure \ref{minspine} (see
\cite[2.4.2]{Matveev-2003}). To construct the theta-curves
$\theta^{(i)}$, $i\in\{ 0, 1, 2, 3\}$, we need to describe certain
cell decompositions of the torus $T=\partial M$ and of its universal
covering $\tilde{T}$. The local embedding $f:T\to P$ determines a
cell decomposition of $T$ as follows.
\begin{enumerate}
    \item The inverse image $f^{-1}(C)$ of every open $k$-dimensional
    cell $C$ of $P$ consists of two open $2$-cells if $k=2$,
    three open arcs if $k=1$, and four points if $k=0$.
    \item The restriction of $f$ to each of these cells
    is a homeomorphism onto the corresponding cell of $P$.
\end{enumerate}

\begin{figure}
\centering
\includegraphics[scale=0.7]{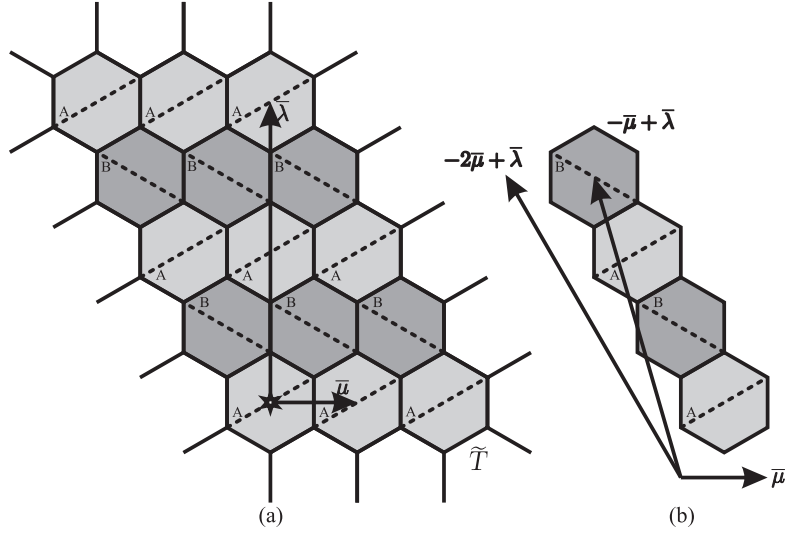}
\caption{Cell decompositions of $\tilde{T}$ (left) and $T$ (right)}
\label{decomposition}
\end{figure}

 Construct the universal covering $\tilde{T}$ of $T$. It can
be presented as a plane decomposed into hexagons, see Fig.
\ref{decomposition}a. The group of covering translations is
isomorphic to the group $\pi_1(T) = H_1(T; \mathbb{Z})$. We choose a
basis $\tilde{\mu}$, $\tilde{\lambda}$ as shown in Fig.
\ref{decomposition}a. It is easy to see that the corresponding
elements of $\pi_1(T)$ (which can be also viewed as oriented loops)
form the canonical coordinate system $(\mu, \lambda)$ on $T$. If we
factor this covering by the translations $\tilde{\mu}$,
$\tilde{\lambda}$, we recover $T$. If we additionally identify the
hexagons marked by the letter $A$ with respect to the composition of
the symmetry in the dotted diagonal of the hexagon and the
translation by $-\tilde{\mu} + \tilde{\lambda}/2$, and do the same
for the hexagons marked by the letter $B$, we obtain $P$. The torus
$T$ is shown in Fig. \ref{decomposition}b as a polygon $D$ composed
of four hexagons. Each side of $D$ is identified with some other one
via the translation along one of the three vectors $\tilde{\mu}$,
$-2\tilde{\mu}+\tilde{\lambda}$, and $-\tilde{\mu}+\tilde{\lambda}$.
The spine $P$ can be presented as the union of two hexagons, see
Fig. \ref{theta0} (right). The edges of the hexagons are decorated
with four different patterns. To recover $P$, one should identify
the edges having the same pattern.

\begin{figure}
\centering
\includegraphics[scale=0.4]{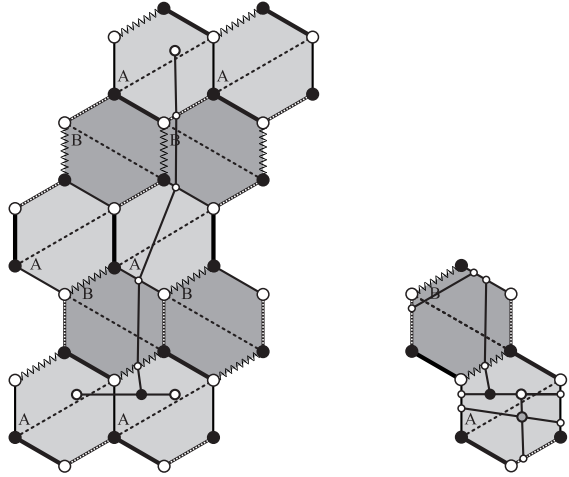}
\caption{The theta-curve $\theta^{(0)}$} \label{theta0}
\end{figure}

 Step 3. Now for each $i\in\{ 0, 1, 2, 3\}$ we exhibit a theta-curve
$\theta^{(i)}\subset \partial M$ such that the simple relative
spine $R(P, \theta^{(i)})$ of $M$ has $10$ interior true vertices
and $\Psi_{\mu, \lambda}(\theta^{(i)}) = \triangle^{(i)}$.

 Consider the wedge of the three arcs on $\tilde{T}$, see Fig.
\ref{theta0} (left). The projections of the arcs onto $T$ yield a
theta-curve that we denoted by $\theta^{(0)}$. It can be checked
directly that $\theta^{(0)}$ is in general position with respect to
the map $f$, and $\Psi_{\mu, \lambda}(\theta^{(0)}) =
\triangle^{(0)}$. It remains to note that the set of the interior
true vertices of $R(P, \theta^{(0)})$ consists of (a) the two true
vertices of the special polyhedron $P$, (b) the images under $f$ of
the two vertices of $\theta^{(0)}$, (c) the five intersection points
of the set $f(\theta^{(0)})$ with the triple lines of $P$, see Fig.
\ref{theta0} (left), and (d) one self-intersection point of the
image $f(\theta^{(0)})$ of $\theta^{(0)}$ (shown in Fig.
\ref{theta0} (right) as a fat gray dot).

\begin{figure}
\centering
\includegraphics[scale=0.4]{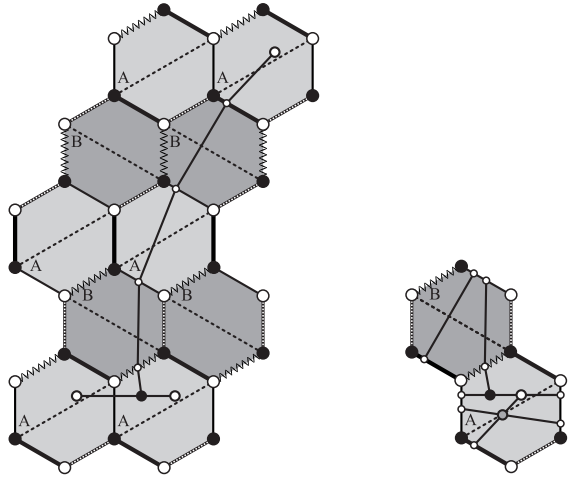}
\caption{The theta-curve $\theta^{(1)}$} \label{theta1}
\end{figure}

\begin{figure}
\centering
\includegraphics[scale=0.4]{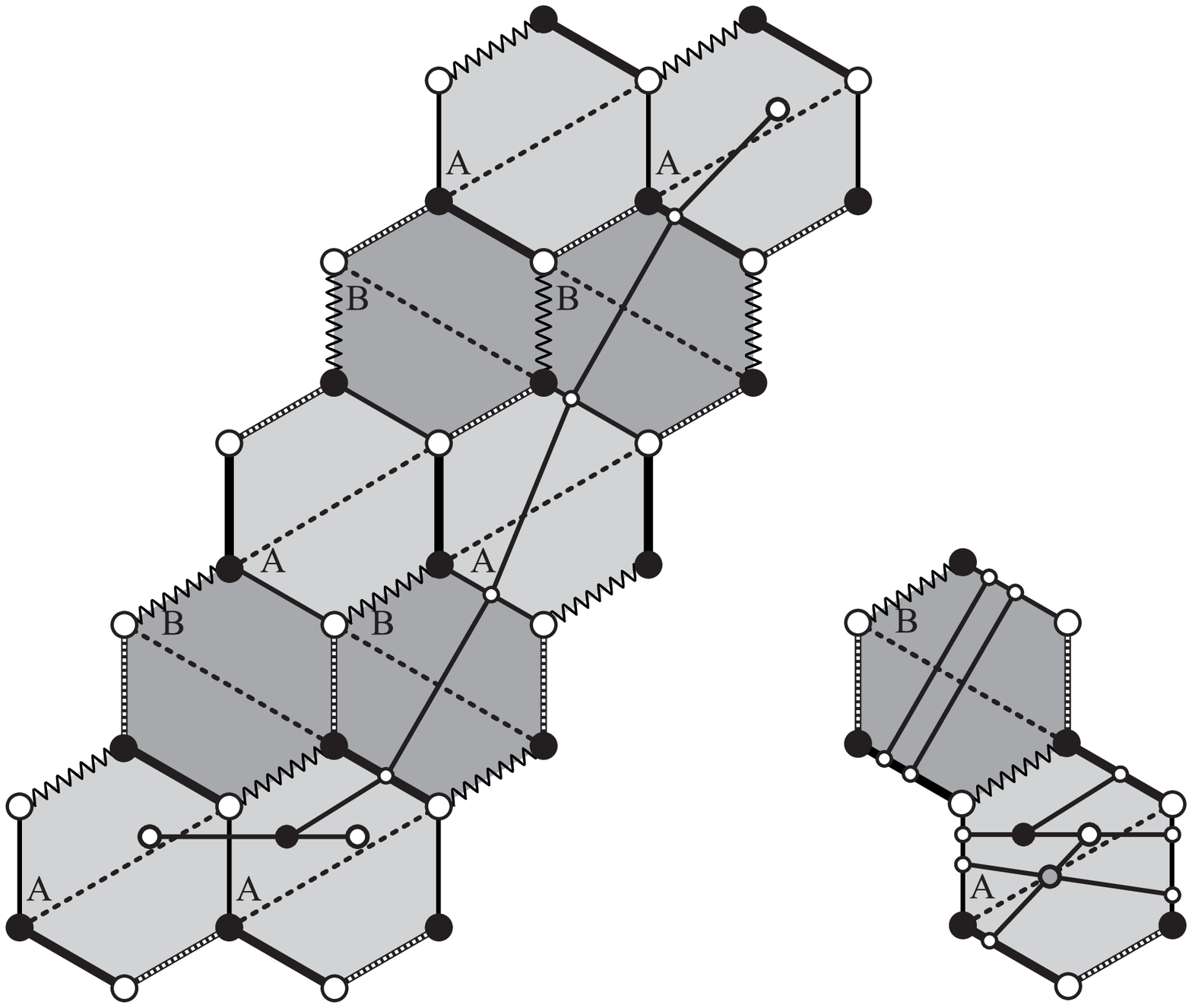}
\caption{The theta-curve $\theta^{(2)}$} \label{theta2}
\end{figure}

\begin{figure}
\centering
\includegraphics[scale=0.4]{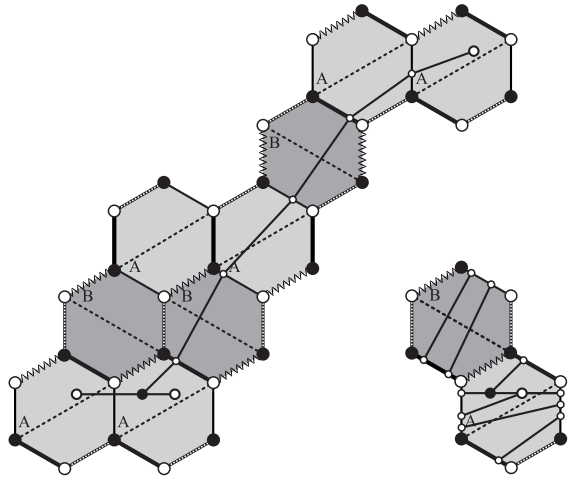}
\caption{The theta-curve $\theta^{(3)}$} \label{theta3}
\end{figure}

 The theta-curves $\theta^{(1)}$, $\theta^{(2)}$, $\theta^{(3)}$
satisfying the conclusion of the Proposition are shown in Fig.
\ref{theta1}, \ref{theta2}, \ref{theta3}. We point out that among
the $10$ interior true vertices of $R(P, \theta^{(3)})$ there are
$6$ intersection points of the set $f(\theta^{(3)})$ with the triple
lines of $P$, see Fig. \ref{theta3} (left), while there are no
self-intersection points of the image $f(\theta^{(3)})$ of
$\theta^{(3)}$, see Fig. \ref{theta3} (right).
\end{proof}


\section{Proof of the main theorem}

 Let $p\geqslant 0$ and $q\geqslant 1$ be two relatively prime integers.
To prove the inequality $c(4_1(p/q))\leqslant \omega(p/q)$ it
suffices to construct a simple spine of the manifold $4_1(p/q)$ with
$\omega(p/q)$ true vertices.

 Thurston \cite{Thurston-1978} proved that the manifold $4_1(p/q)$ is
hyperbolic except for $p/q\in \{0, 1, 2, 3, 4, \infty\}$. The case
$p/q=\infty$ does not satisfy the assumptions of the Theorem. In
each of the five remaining cases the non-hyperbolic manifold
$4_1(p/q)$ has complexity $7$ and $\omega(p/q)=7$.

 Let us construct a simple spine of the hyperbolic manifold
$4_1(p/q)$. Recall that the meridian $m$ and the theta-curve
$\theta_m$ on the boundary of $(V, \theta_V)$ were fixed in Example
1. Let $(\mu, \lambda)$ be the canonical coordinate system on the
boundary torus $\partial E(4_1)$ of the figure eight knot complement
$E(4_1)$. Among all homeomorphisms $\partial V \to
\partial E(4_1)$ that take $m$ to the curve $\mu^p\lambda^q$, we
choose a homeomorphism $\varphi$ such that the distance between the
theta-curves $\varphi(\theta_m)$ and $\theta^{(0)}$ be as small as
possible. For convenience denote by $z$ the number $\min\{[p/q],
3\}$. By the Proposition, the manifold $(E(4_1), \theta^{(z)})$ has
a simple relative spine with $10$ interior true vertices. Since
$4_1(p/q) = V\cup_\varphi E(4_1)$, it follows from Lemma that the
manifold $(4_1(p/q), \emptyset)$ has a simple relative spine
$Q_{p/q}$ with $10 + d(\varphi(\theta_V), \theta^{(z)})$ interior
true vertices. Moreover, $Q_{p/q}$ is a spine of $4_1(p/q)$, since
$\partial 4_1(p/q) = \emptyset$.

 Now let us prove that $d(\varphi(\theta_V), \theta^{(z)}) = -2 +
\max\{[p/q]-3, 0\} + S(rem(p,q),q)$. Recall that for each $i\in\{ 0,
1, 2, 3\}$ the map $\Psi_{\mu, \lambda}$ takes $\theta^{(i)}$ to the
triangle $\triangle^{(i)}$ of the Farey triangulation with the
vertices at $i$, $i+1$, and $\infty$. Denote by $\triangle_V$ and
$\triangle_m$ the triangles $\Psi_{\mu, \lambda}(\varphi(\theta_V))$
and $\Psi_{\mu, \lambda}(\varphi(\theta_m))$, respectively. Since
the distance between theta-curves on $\partial E(4_1)$ is equal to
the distance between the corresponding triangles of $\mathbb{F}$, it
is sufficient to find $d(\triangle_V, \triangle^{(z)})$.

 The choice of $\varphi$ guarantees us that $\triangle_m$ is
the closest triangle to $\triangle^{(0)}$ among all the triangles
with a vertex at $p/q$. This implies (see \cite[Proposition
4.3]{Martelli-Petronio-2004} and \cite[Lemma 2]{Fominykh-2008}) that
$d(\triangle_m, \triangle^{(0)}) = S(p,q)-1$. Since the theta-curve
$\theta_V$ is obtained from $\theta_m$ by exactly one flip and
$\theta_V$ does not contain the meridian $m$, the triangle
$\triangle_V$ has a common edge with $\triangle_m$ and $p/q$ is not
a vertex of $\triangle_V$. Hence, $d(\triangle_V, \triangle^{(0)}) =
S(p,q)-2$. Analyzing the Farey triangulation, we can notice that
$d(\triangle_V, \triangle^{(z)}) = d(\triangle_V, \triangle^{(0)}) -
d(\triangle^{(z)}, \triangle^{(0)})$. Taking into account that
$d(\triangle^{(z)}, \triangle^{(0)})=z$, $S(p,q) = [p/q] +
S(rem(p,q),q)$ and $[p/q] - \min\{[p/q], 3\} = \max\{[p/q]-3, 0\}$
we get the equality $d(\varphi(\theta_V), \theta^{(z)}) =
d(\triangle_V, \triangle^{(z)}) = -2 + \max\{[p/q]-3, 0\} +
S(rem(p,q),q)$.

 Note that if $p/q\not\in \mathbb{Z}$, then $Q_{p/q}$ is the desired
spine, since it contains $\omega(p/q)$ true vertices. On the other
hand, if $p/q\in \mathbb{Z}$, the spine $Q_{p/q}$ contains
$\omega(p/q)+1$ true vertices. In this case $Q_{p/q}$ can be
transformed into another simple spine $Q_{p/q}'$ of $4_1(p/q)$ by a
sequence of moves along boundary curves of length $4$ (similar
arguments can be found in \cite[page 81]{Matveev-2003}). The spine
$Q_{p/q}'$ has the same number of true vertices but possesses a
boundary curve of length $3$, hence it can be simplified. The result
is a new spine of $4_1(p/q)$ with $\omega(p/q)$ true vertices.

 To conclude the proof of the theorem, it remains to note that the
table \cite{Atlas}contains $46$ hyperbolic manifolds of the type
$4_1(p/q)$ satisfying $\omega(p/q)\leqslant 12$. For each of them
our upper bound is sharp.


\end{document}